\documentclass[a4paper,12pt]{article}

\usepackage[table,dvipsnames,svgnames,x11names,hyperref]{xcolor}
\usepackage[utf8]{inputenc}

\usepackage{lmodern}
\usepackage[T1]{fontenc}

\usepackage[top=1in, bottom=1.25in, left=1.25in, right=1.25in]{geometry}
\usepackage{amsmath,amssymb,amsthm}

\usepackage{tikz}
\usepackage{array}
\usepackage{enumerate}
\usepackage{graphicx}
\usepackage{float}
\usepackage{caption}
\usepackage{subcaption}
\usepackage[colorlinks=true, allcolors=MidnightBlue]{hyperref}

\newtheorem{thm}{Theorem}[section]

\newtheorem{lem}[thm]{Lemma}

\theoremstyle{definition}

\theoremstyle{remark}

\newtheorem{question}{Question}
\newtheorem*{question1p}{Question~$1'$}
\newtheorem*{question2p}{Question~$2'$}

\newtheorem*{example}{Example}

\author{Sel\c{c}uk Kayacan  \thanks{I would like to thank to Roald Koudenburg for helpful discussions.}}
\title{Recovering information about a finite group\\
  from its subrack lattice}
\date{}

\begin{document}

\maketitle

\small

\begin{center}
  Mathematics Research Group,\\
  METU Northern Cyprus Campus,\\
  Kalkanl\i, G\"uzelyurt, Turkey\\
  {\it e-mail:} \href{mailto:skayacan@metu.edu.tr}{skayacan@metu.edu.tr}
\end{center}

\begin{abstract}
  We prove that the isomorphism type of the subrack lattice of a finite group determines the nilpotence class. We analyze the problem of estimating the orders of the group elements corresponding to the atoms of the subrack lattice. As a result, we show that the subrack lattice determines $p$-nilpotence of the group if a certain condition is met.
  


\end{abstract}

\section{Introduction}

What information can be obtained from a group if we forget about the group operation but preserve the conjugation operation? This question manifests itself naturally in the study of racks. A \emph{rack} $R$ is a set (possibly empty) together with a binary operation $\triangleright\colon R\times R\to R$ satisfying the following two axioms:
\begin{itemize}
\item[\textbf{(A1)}] for all $a,b,c\in R$ we have $a \triangleright (b \triangleright c) = (a \triangleright b) \triangleright (a \triangleright c)$ 
\item[\textbf{(A2)}] for all $a,c\in R$ there is a unique $b\in R$ such that $a \triangleright b = c$
\end{itemize}
If, additionally, $R$ satisfies the following third axiom
\begin{itemize}
\item[\textbf{(A3)}] for all $a\in R$ we have $a\triangleright a = a$
\end{itemize}
we say $R$ is a \emph{quandle}. Let $G$ be a finite group and for any $a,b\in G$ let $a\triangleright b := aba^{-1}$. The prime example of a rack in this paper is the group $G$ together with the conjugation operation $\triangleright$. In fact the group $G$ is a quandle, since axiom (A3) holds as well under the conjugation operation.

Racks appeared in literature under various names in different times. One of the oldest formalizations, as early as in 1950s, is due to Conway and Wraith in which they used the term ``wrack''. The choice of this term makes it apparent that those objects may arise as the ruins of a group. In \cite{FR92} Fenn and Rourke adopted this term by dropping the first letter ``w'' and their usage became prevalent.

In a recent work \cite{HSW19} Heckenberger, Shareshian, and Welker initiated the study of racks from the combined perspective of combinatorics and finite group theory. A \emph{subrack} of a rack $R$ is a subset of $R$ which is a rack under the operation inherited by $R$. The set of all subracks of $R$ form a partially ordered set (poset for short) under inclusion which is denoted by $\mathcal{R}(R)$. It is a simple fact that for a finite rack $R$ the partially ordered set $\mathcal{R}(R)$ is a lattice (see \cite[Lemma~2.1]{HSW19}). We define the \emph{subrack lattice} of the group $G$ as the partially ordered set $\mathcal{R}(G)$ with the maximum element $G$ and the minimum element $\emptyset$. In \cite[Question~5.2]{HSW19} the authors raised the following

\begin{question}\label{q1}
  Are there two groups $G$, $H$ which have isomorphic subrack lattices but are non-isomorphic as racks?
\end{question}

We shall remark that the rack structure on the group $G$ determines the group structure of the inner automorphism group $G/Z(G)$ and answering Question~\ref{q1} negatively would be a strong result. In particular, main theorems of this paper would become corollaries.

Let $w_G\colon \mathbb{N}\to \mathbb{Z}$ be the \emph{class size frequency function} taking a number $n$ to the the number of conjugacy classes of $G$ with $n$ elements. In \cite{CHM92} Cossey, Hawkes, and Mann proved that if $G$ is a nilpotent group, and if $H$ is a group with $w_G = w_H$, then $H$ is nilpotent. However, they were not able to decide whether the theorem holds if nilpotence is replaced by supersolvability. By \cite[Lemma~2.8]{HSW19} the subrack lattice $\mathcal{R}(G)$ determines the class size frequency function $w_G$. Hence, in view of subrack lattices, we may ask the following

\begin{question}\label{q2}
  If $G$ is a supersolvable group, and if $H$ is a group with $\mathcal{R}(G)\cong \mathcal{R}(H)$, can we conclude $H$ is supersolvable?
\end{question}

In \cite[Theorem~1.1]{HSW19} the authors answered this problem affirmatively by proving that the isomorphism type of the subrack lattice $\mathcal{R}(G)$ determines if $G$ is abelian, nilpotent, supersolvable, solvable, or simple.

In \cite{CHM92} the authors introduced some examples showing that the nilpotence class of the group $G$ is not determined by the class size frequency function $w_G$.

\begin{question}\label{q3}
  Is the nilpotence class of the group $G$ determined by the isomorphism type of the subrack lattice $\mathcal{R}(G)$?
\end{question}

In Section~\ref{sec:nil} we give a positive answer to this question (see Theorem~\ref{thm:nil}). We shall remark that character table also determines the nilpotence class. In \cite{Mat94} Mattarei constructed two non-isomorphic groups with identical character tables and different derived length. Interestingly the inner automorphism groups of those two groups are isomorphic by the construction.

Let $m_G\colon \mathbb{N}\to \mathbb{Z}$ be the \emph{character degree frequency function} taking a number $n$ to the number of irreducible characters of $G$ of degree $n$. In \cite{Isa86} Isaacs proved that if $G$ has a normal $p$-complement for some prime $p$, and if $H$ is a group with $m_G = m_H$, then $H$ has a normal $p$-complement. In view of subrack lattices we may ask the following

\begin{question}\label{q4}
  Is the $p$-nilpotence of the group $G$ determined by the isomorphism type of the subrack lattice $\mathcal{R}(G)$?  
\end{question}

In Section~\ref{sec:nor} we analyze the problem of estimating the orders of group elements using only the knowledge of subrack lattice of the group. As a result of our analysis we answer this question affirmatively when the subrack lattice satisfies a certain condition (see  Theorem~\ref{thm:pnil}).

\section{Preliminaries}
\label{sec:pre}

For a finite group $G$ there are two structures over $G$ that we consider. One is the usual multiplicative structure of the group and the other is the conjugation structure. If $S$ is a subset of $G$ we denote by $\langle S\rangle$ the subgroup of $G$ generated by $S$ whereas $\langle S\rangle_{\mathsf{rk}}$ stands for the subrack of $G$ generated by $S$. If $S$ and $T$ are two subsets of a finite rack $R$, we define
$$ S\triangleright T := \{ s\triangleright t\colon s\in S,\, t\in T \}. $$
As was remarked earlier the group $G$ is a quandle under the conjugation operation; hence, singletons are elements of the subrack lattice $\mathcal{R}(G)$ and the order of $G$ is the number of atoms in $\mathcal{R}(G)$. 

\begin{lem}\label{lem:subrack}
  For a finite group $G$ and a subset $S$ of $G$ the following statements are equivalent.
  \begin{enumerate}[(i)]
  \item The set $S$ is a subrack of $G$.
  \item The equality $S = S\triangleright S$ holds.
  \item The equality $S = \langle S \rangle \triangleright S$ holds.
  \item For any pair of elements $a,b\in S$, the subrack $\langle a,b \rangle_{\mathsf{rk}}$ is contained by $S$.
  \end{enumerate}
\end{lem}

\begin{proof}
  Obviously (i) implies (ii) and (iii) implies (i). To show that (ii) implies (iii) observe that $a \triangleright (b \triangleright c) = (ab) \triangleright c$ for all $a,b,c\in G$. This shows the equivalence of the statements (i), (ii), and (iii). Since $a \triangleright b\in \langle a,b \rangle_{\mathsf{rk}}$, the equivalence of (i) and (iv) is clear.
\end{proof}

Notice that any subgroup of a group is a subrack by Lemma~\ref{lem:subrack}; hence, subgroup lattice is contained in the subrack lattice of the group. Consider the subposet $\mathcal{S}$ of $\mathcal{R}(G)$ whose elements are atoms of $\mathcal{R}(G)$ together with the subracks that are generated by pairs of elements of $G$. By the equivalence of the statements (i) and (iv) in Lemma~\ref{lem:subrack} we see that given $\mathcal{S}$ one can easily recover the rest of $\mathcal{R}(G)$. The following Lemma will be very useful throughout this paper.

\begin{lem}[see {\cite[Lemma~2.8]{HSW19}}]\label{lem:2.8}
  For a finite group $G$, the maximal subracks of $G$ are unions of all but one conjugacy classes of $G$.
\end{lem}

For a lattice $\mathcal{L}$ the subposet $\mathsf{Int}(\mathcal{L})$ of $\mathcal{L}$ consists of those elements that are the meet of a set of coatoms of $\mathcal{L}$. For an element $S$ of $\mathcal{L}$ we define the \emph{closure} $\overline{S}$ of $S$ as the meet of the coatoms of $\mathcal{L}$ that are covering $S$. We say $S$ is \emph{closed} if $\overline{S} = S$. Let $S$ and $T$ be some elements of $\mathcal{L}$. The \emph{interval} $[S,T]_{\mathcal{L}}$ is defined as the set $\{U\in\mathcal{L} \colon S\leq U \leq T\}$.

\begin{lem}[see {\cite[Lemma~3.2]{HSW19}}]\label{lem:3.2}
  For a finite group $G$ and an element $S$ of $\mathcal{R}(G)$ the following statements are equivalent.
  \begin{enumerate}[(i)]
  \item The subrack $S$ of $G$ is closed in $\mathcal{R}(G)$.
  \item The subrack $S$ lies in $\mathsf{Int}(\mathcal{R}(G))$.
  \item The subrack $S$ is a union of some of the conjugacy classes of $G$.
  \end{enumerate}
\end{lem}

Let $\mathcal{L}$ be a lattice with the maximum element $\hat{1}$ and minimum element $\hat{0}$ that is isomorphic to the subrack lattice $\mathcal{R}(G)$. The information encoded in $\mathcal{L}$ is not always enough to determine which elements of $\mathcal{L}$ are corresponding to a subgroup of the group $G$ up to an isomorphism between $\mathcal{R}(G)$ and $\mathcal{L}$. For example, if $G$ is an abelian group of order $p^2$ we cannot determine whether the exponent of $G$ is $p$ or $p^2$ by looking at the isomorphism type of $\mathcal{L}$. On the other hand, if the group $G$ is non-abelian we can argue that some particular elements of $\mathcal{L}$ would be the image of a subgroup of $G$ under \emph{any} isomorphism between $\mathcal{R}(G)$ and $\mathcal{L}$. As an example, following \cite{HSW19}, we may define $\mathsf{M}(\mathcal{L})$ as the set of all $S\in \mathcal{L}$ satisfying all of
\begin{enumerate}
\item $S$ is not closed in $\mathcal{L}$,
\item $\{T\in \mathcal{L}\colon S\leq T \text{ and } \overline{S}\nleq T\} = \{S\}$,
\item if $\overline{S}\leq T$ then $T$ is closed in $\mathcal{L}$,
\item $\mathsf{Int}([\hat{0},\overline{S}]_{\mathcal{L}})$ is not a Boolean algebra.
\end{enumerate}

Following Lemma will be useful in proving Theorem~\ref{thm:pnil}.

\begin{lem}[see {\cite[Lemma~3.4]{HSW19}}]\label{lem:3.4}
  Let $G$ be a finite group and $f\colon \mathcal{R}(G) \to \mathcal{L}$ be a lattice isomorphism. The following statements hold.
  \begin{enumerate}
  \item If $S\in \mathsf{M}(\mathcal{L})$, then $f^{-1}(S)$ is a non-normal subgroup of $G$.
  \item If $M$ is a non-normal maximal subgroup of $G$, then $f(M)$ is an element of $\mathsf{M}(\mathcal{L})$.
  \end{enumerate}
\end{lem}

In the rest of this paper we will implicitly identify the lattice $\mathcal{L}$ with the subrack lattice $\mathcal{R}(G)$ so that we won't need to refer to $\mathcal{L}$ anymore. To be more precise, we fix a set $L$ whose elements are labels for the atoms of $\mathcal{L}$ as well as for the corresponding atoms of $\mathcal{R}(G)$ under the isomorphism. Further, by abuse of notation, we identify $G=L$ so that the elements of $L$ can be manipulated according to the group operation.

For a subset $S$ of a group $G$ we define $K_G(S)$ as the union of the conjugacy classes of elements of $S$ in the group $G$. When $S=\{a\}$ is a singleton we write $K_G(a) := K_G(S)$ for simplicity. The following auxiliary result would be instrumental in proving Theorem~\ref{thm:nil}.

\begin{lem}\label{lem:int1}
  Let $G$ be a finite group and $a,b$ be elements of $G$. If $b = az$, where $z$ is a central element of $G$, then the permutation $\pi$ of $G$ interchanging $K_G(a)$ with $K_G(b)$ by mapping $gag^{-1}$ to $gbg^{-1}$ and vice versa for all $g\in G$ and fixing the rest of the elements of $G$ induces an automorphism of $\mathcal{R}(G)$.
\end{lem}

\begin{proof}
  Clearly, the mapping $\pi$ is a permutation of $G$. Let $S$ and $T$ be subracks of $G$. By assumption $S\triangleright S = S$. Since $(gag^{-1})\triangleright x = y$ if and only if $(gbg^{-1})\triangleright xz = yz$ for every $g,x,y\in G$, we see that $\pi(S)\triangleright \pi(S) = \pi(S)$. In other words, the image of the subrack $S$ under $\pi$ is a subrack. A similar argument shows that $x\in S\vee T$ if and only if $\pi(x)\in \pi(S)\vee \pi(T)$, so that $\pi(S\vee T) = \pi(S)\vee \pi(T)$. Finally, since $S\wedge T = S\cap T$, the equality $\pi(S\wedge T) = \pi(S)\wedge \pi(T)$ holds as well.
\end{proof}

Notice that by Lemma~\ref{lem:int1} any pair of elements $a$ and $b$ of $G$ that are lying in the same coset of $Z(G)$ are indistinguishable from each other considered as elements of $\mathcal{R}(G)$. Additionally, the conjugacy classes $K_G(a)$ and $K_G(b)$ of those elements are also indistinguishable in $\mathcal{R}(G)$.

We conclude this section by proving another statement which is very similar to Lemma~\ref{lem:int1}. It would be useful to prove Theorem~\ref{thm:pnil}.

\begin{lem}\label{lem:int2}
  Let $G$ be a finite group and $a,b$ be elements of $G$. If the subgroup generated by $a$ is same with the subgroup generated by $b$, then the permutation $\pi$ of $G$ interchanging $K_G(a)$ with $K_G(b)$ by mapping $gag^{-1}$ to $gbg^{-1}$ and vice versa for all $g\in G$ and fixing the rest of the elements of $G$ induces an automorphism of $\mathcal{R}(G)$.
\end{lem}

\begin{proof}
  Let $S$ and $T$ be subracks of $G$. By assumption $\langle S\rangle \triangleright S = S$. Since $a$ and $b$ generates the same subgroup, we see that $b = a^k$ for some integer $k$ and $\pi(\langle S \rangle) = \langle S \rangle$. Observe that $x\triangleright gag^{-1} = hah^{-1}$ if and only if $x\triangleright gbg^{-1} = hbh^{-1}$ for every $x,g,h\in G$, so that  $\pi(S)\triangleright \pi(S) = \pi(S)$. Similar reasoning shows that $\pi(S\vee T) = \pi(S)\vee \pi(T)$.
\end{proof}

\section{Nilpotence class of the group}
\label{sec:nil}

It is easy to determine which elements of $\mathcal{R}(G)$ would constitute a maximal abelian subgroup of $G$. Let $\mathsf{A}(G)$ be the set of maximal elements $A\in \mathcal{R}(G)$ having the property that the interval $[\emptyset,A]_{\mathcal{R}(G)}$ is a Boolean algebra. Let $\mathsf{N}(G)$ be the set of elements $N\in \mathcal{R}(G)$ formed in the following way. Each element $N\in \mathsf{N}(G)$ is the union of the conjugacy classes of $G$ that are contained by some fixed $A\in \mathsf{A}(G)$. Notice that the combinatorial structure of $\mathcal{R}(G)$ is enough to determine both $\mathsf{A}(G)$ and $\mathsf{N}(G)$.

\begin{lem}\label{lem:ab}
  Let $G$ be a non-abelian finite group. The following statements hold.
  \begin{enumerate}
  \item An element $A\in \mathcal{R}(G)$ is a maximal abelian subgroup of $G$ if and only if the subrack $A$ lies in the set $\mathsf{A}(G)$.
  \item An element $N\in \mathcal{R}(G)$ is a maximal normal abelian subgroup of $G$ if and only if the subrack $N$ lies in the set $\mathsf{N}(G)$.
  \end{enumerate}
\end{lem}

\begin{proof}
  \emph{Statement~1.} Suppose $A$ is a maximal abelian subgroup of $G$. Clearly, the interval $[\emptyset,A]_{\mathcal{R}(G)}$ is a Boolean algebra. Let $B$ be a maximal element of $\mathcal{R}(G)$ with the property $[\emptyset,B]_{\mathcal{R}(G)}$ is a Boolean algebra. Then the elements of $G$ corresponding to the atoms of $[\emptyset,B]_{\mathcal{R}(G)}$ generate an abelian subgroup $\langle B\rangle$ containing $B$. Since $[\emptyset,B]_{\mathcal{R}(G)}$ is a Boolean algebra and it is a maximal element in $\mathcal{R}(G)$ with this property, we see that $\langle B\rangle = B$. Next suppose $B$ contains $A$. Since $G$ is non-abelian, the subrack $B$ is a proper subset of $G$. Also, since $B$ is a subgroup which is abelian and $A$ is a maximal abelian subgroup of $G$ contained by $B$, we see that $A=B$.

  Next, suppose $A$ is a maximal element of $\mathcal{R}(G)$ with the property that the interval $[\emptyset,A]_{\mathcal{R}(G)}$ is a Boolean algebra. As was explained in the previous paragraph $A$ is an abelian subgroup of $G$ which is maximal.

  \emph{Statement~2.} Observe that the core of $A\in \mathcal{R}(G)$ is just the union of the conjugacy classes of $G$ that are contained by $A$.
\end{proof}

Two elements of a group commutes with each other if and only if the join of the corresponding atoms in the subrack lattice covers no other elements. Therefore, the information of the subrack lattice of a finite group can be used to determine the centralizers of the subsets of the group. In particular, we can locate the center of the group in the subrack lattice.

A representation of a finite group is monomial if it is induced from a linear representation of a subgroup. A group whose all complex irreducible representations are monomial is called a \emph{monomial} group. The class of monomial groups lies between the classes of supersolvable and solvable groups. By \cite[Theorem~1.1]{HSW19} the isomorphism type of the subrack lattice determines if the group is supersolvable or solvable. Thus, it is natural to ask the following

\begin{question2p}
  If $G$ is a monomial group, and if $H$ is a group with $\mathcal{R}(G)\cong \mathcal{R}(H)$, can we conclude $H$ is monomial?
\end{question2p}

We don't know how to approach this problem; nonetheless, we shall make a remark. In proving supersolvable groups are monomial an important ingredient is the following property of supersolvable groups: There exist an abelian normal subgroup of the group which is not central. By our previous discussions we see that this property can easily be recognized with the information provided by the subrack lattice of the group.

\begin{lem}\label{lem:coset}
  Let $G$ be a finite group and $N$ be a normal subgroup of $G$. Then any coset of $N$ is a subrack of $G$. Moreover, for any subset  $\{a_1N,\dots,a_kN\}$ of $G/N$, the join of those cosets in $\mathcal{R}(G)$ is the union of the cosets of $N$ containing $\langle a_1,\dots,a_k \rangle_{\mathsf{rk}}$.
\end{lem}

\begin{proof}
 Let $a$ be an element of $G$. Since $an_1\triangleright an_2 = an_1an_2n_1^{-1}a^{-1}\in aN$ for every $n_1,n_2\in N$, we see that the coset $aN$ is a subrack of $G$. For the latter part of the claim, observe that $a_1N\triangleright a_2N = (a_1\triangleright a_2)N$.
\end{proof}

Let $\mathcal{C}$ be a collection of subracks of $G$ and let $\mathsf{J}(\mathcal{C})$ be the subposet of $\mathcal{R}(G)$ formed in the following way. A subrack $S$ of $G$ is an element of $\mathsf{J}(\mathcal{C})$ if and only if $S$ is the union of some of the elements of $\mathcal{C}$ in $\mathcal{R}(G)$. Let $N$ be a normal subgroup of $G$ and let $\mathcal{C}$ be the partition of $G$ into the cosets of $N$. By Lemma~\ref{lem:coset} we see that the poset $\mathsf{J}(\mathcal{C})$ is isomorphic to the poset $\mathcal{R}(G/N)\setminus \{\emptyset\}$. Let $\mathcal{C}$ be a partition of $G$ into subsets each of which is a subrack of size $|N|$ such that $N\in \mathcal{C}$ and the join of any subset of $\mathcal{C}$ in $\mathcal{R}(G)$ is a union of some of the elements of $\mathcal{C}$. Unfortunately, those premises are not enough to conclude that $\mathsf{J}(\mathcal{C}) \cong \mathcal{R}(G/N)\setminus \{\emptyset\}$. However, we can determine an isomorphic copy of $\mathcal{R}(G/Z(G))\setminus \{\emptyset\}$ in $\mathcal{R}(G)$.

\begin{thm}\label{thm:nil}
  If $G$ is a finite nilpotent group of class $c$, and if $H$ is a group with $\mathcal{R}(G)\cong \mathcal{R}(H)$, then $H$ is nilpotent of class $c$. 
\end{thm}

\begin{proof}
  First we shall construct a poset which is isomorphic to $\mathcal{R}(G/Z(G))\setminus \{\emptyset\}$. Let $L$ be the set of atoms of $\mathcal{R}(G)$ which we identify with the elements of $G$. Let $\mathcal{C}$ be a partition of $L = G$ having the following properties:
  \begin{enumerate}[(i)]
  \item Each element of $\mathcal{C}$ has cardinality $|Z(G)|$.
  \item The join of any subset of $\mathcal{C}$ in $\mathcal{R}(G)$ is a union of some of the elements of $\mathcal{C}$.
  \item For any two distinct elements $a$ and $b$ in $G$ that are lying in the same element of $\mathcal{C}$, there exists a permutation $\pi$ of $G$ interchanging $a$ with $b$, also interchanging $K_G(a)$ with $K_G(b)$ as sets, and fixing the rest of the elements of $G$, so that $\pi$ induces an automorphism of $\mathcal{R}(G)$ preserving the partition $\mathcal{C}$.
  \end{enumerate}
  By Lemma~\ref{lem:int1} we know that the partition of $G$ into the cosets of $Z:=Z(G)$ satisfies those above conditions. In particular, if $a$ and $az$ are two elements lying in the same coset of $Z$, then the permutation $\pi$ which swaps the pairs $gag^{-1}$ and $gag^{-1}z$ for every $g\in G$ and fixes the rest of the elements woulds yield an automorphism of $\mathcal{R}(G)$.
  
  Let $f\colon \mathcal{C}\to G/Z$ be a bijection having the property that for each $C\in \mathcal{C}$ the intersection $C\cap f(C)$ is non-empty. Let $\{a_1Z,\dots,a_kZ\}$ be a subset of $G/Z$ and let $C_{a_i} := f^{-1}(a_iZ)$, $1\leq i\leq k$. Here we require the transversal elements are chosen in a way that $a_i$ lies in $C_{a_i}$ for each $1\leq i\leq k$. Now, the join of the subracks $C_{a_1},\dots,C_{a_k}$ in $\mathcal{R}(G)$ is the union of the elements of $\mathcal{C}$ containing $\langle a_1,\dots,a_k \rangle_{\mathsf{rk}}$. To see this, observe that for any pair $a_i'\in C_{a_i}$ and $a_j'\in C_{a_j}$ there exists an automorphism of $\mathcal{R}(G)$ fixing $\langle C_{a_i},C_{a_j}\rangle_{\mathsf{rk}}$ that is induced by a permutation of $G$ mapping $a_i$ to $a_i'$ and $a_j$ to $a_j'$. Then, by Lemma~\ref{lem:coset}, the map $f$ induces an isomorphism from $\mathsf{J}(\mathcal{C})$ to ${\mathcal{R}(G/Z(G))\setminus \{\emptyset\}}$.

  Since $\mathsf{J}(\mathcal{C})$ is defined purely by the combinatorial properties of $\mathcal{R}(G)$ and since $\mathcal{R}(G)\cong \mathcal{R}(H)$ by assumption, we see that $\mathsf{J}(\mathcal{C})$ and ${\mathcal{R}(H/Z(H))\setminus \{\emptyset\}}$ are isomorphic posets. As we know, the nilpotence class of the group is equal to the length of the upper central series and by iterating the above process one would get a series of posets whose $k$th term isomorphic to $\mathcal{R}(G/Z_k)\setminus \{\emptyset\}$, where $Z_k$ is the $k$th center of $G$. By assumption $G$ is nilpotent of class $c$; thus, the series stabilizes when $k = c$ and the corresponding poset would be a singleton. Since this must be the case for the group $H$ as well, we conclude that $H$ is a nilpotent group with nilpotence class $c$.
\end{proof}

\section{$p$-Nilpotence of the group}
\label{sec:nor}

For a finite rack $R$, by axiom (A2), the map $\phi_a\colon a\mapsto a\triangleright b$ defines a permutation of $R$. Moreover, by axiom (A1), the map $\phi_a$ is an automorphism of $R$. Let $\Phi\colon R\to S_R$ be the map taking $a\in R$ to $\phi_a\in S_R$, where $S_R$ is the group of permutations of the set $R$. The \emph{inner automorphism group} $\mathsf{Inn}(R)$ of $R$ is the subgroup of $S_R$ generated by the elements of $\Phi(R)$. Observe that $\mathsf{Inn}(R)$ and $\Phi(R)$ coincides if $R=G$ is a finite group.

\begin{lem}
  Let $G$ be a finite group. The following statements hold.
  \begin{enumerate}
  \item The map $\Phi\colon G\to \Phi(G)$ is a group homomorphism whose kernel is $Z(G)$.
  \item The map $\Phi\colon G\to \Phi(G)$ is a rack homomorphism.
  \end{enumerate}
\end{lem}

\begin{proof}
  See the proof of \cite[Proposition~2]{Kay20}.
\end{proof}

Let $G$ be a centerless group. By the above Lemma the groups $G$ and $\Phi(G)$ are isomorphic. Given the subrack lattice $\mathcal{R}(G)$ we want to estimate the orders of the elements of $\Phi(G)$. Let $S, T$ be subracks of the group $G$. As a notational shorthand, we introduce the set
$$ B_G(S,T) := \langle S\cup T \rangle_{\mathsf{rk}} \cap K_G(T). $$
If $S = \{a\}$ (or $T = \{g\}$) is a singleton, for simplicity we use $B_G(a,T) := B_G(S,T)$ (or $B_G(S,g) := B_G(S,T)$). Notice that knowledge of the subrack lattice is enough to determine $B_G(S,T)$ for any pair of subracks $S$ and $T$ of $G$.

For an element $\phi_a\in \Phi(G)$, the \emph{hypothetical pseudo cycle form} $[\phi_a']$ of $\phi_a$ is the partition of $G$ formed by taking the differences of the sets $B_G(a,g)$, $g\in G$. (To be more precise, $[\phi_a']$ is the partition of the set $L$ of labels of the atoms of $\mathcal{R}(G)$ so as to be associated with the group elements.) We call each element of this partition a \emph{part} of $[\phi_a']$. Notice that any part of $[\phi_a']$ is the union of the underlying sets of some of the cycles in the cycle decomposition of $\phi_a$. Imitating the cycle decomposition of $\phi_a$, we represent $[\phi_a']$ by writing down each of its parts side by side within brackets. If $[g_1,\dots,g_s]$ is a part of $[\phi_a']$ such that $[\emptyset,\{g_1,\dots,g_s\}]_{\mathcal{R}(G)}$ is a Boolean algebra in $\mathcal{R}(G)$, we mark it with double brackets as $[[g_1,\dots,g_s]]$. Clearly, any part of $[\phi_a']$ which is a singleton would be marked with double brackets. More generally, if we know that the cardinality of a part of $[\phi_a']$ is the length of the corresponding cycle in $\phi_a$, we call it a \emph{cycle part} of $[\phi_a']$ and mark it with double brackets.

\begin{example}
  Given the subrack lattice of the symmetric group $S_3$, let us estimate the orders of the elements of $S_3$ corresponding to the atoms of the subrack lattice. Let $L = \{a,b,c,d,e,f\}$ be the set of labels of the atoms of $\mathcal{R}(S_3)$ so that the conjugacy classes of $L=S_3$ are
  $$  \{a\},\quad \{b,c,d\},\; \text{ and }\; \{e,f\}. $$
  At this point we can easily determine which element of $S_3$ has what order, since up to isomorphism $S_3$ is the only non-abelian group of order six. Let us ignore this fact and proceed systematically. Clearly, the element $a$ is the identity element of the group. Checking the subrack lattice $\mathcal{R}(S_3)$ we see that $B_G(b,b) = \{b\}$, $B_G(b,c) = B_G(b,d) = \{b,c,d\}$, and $B_G(b,e) = B_G(b,f) = \{e,f\}$, so that
  $$ [\phi_b'] = [[a]]\,[[b]]\,[[c,d]]\,[[e,f]].  $$
  Similarly, we can determine $[\phi_c']$ and $[\phi_d']$. To determine $[\phi_e']$ and $[\phi_f']$, by checking $\mathcal{R}(S_3)$, we observe that $e$ and $f$ commutes and
  $$ B_G(e,b) = B_G(f,b) = B_G(e,c) = B_G(f,c) = B_G(e,d) = B_G(f,d) = \{b,c,d\}. $$
  Since there must exist at least two elements of $S_3$ of order three and since $e$ and $f$ are the only possibilities, we see that
  $$ [\phi_e'] = [\phi_f'] = [[a]]\,[[b,c,d]]\,[[e]]\,[[f]].  $$
Notice that we were unable to distinguish the elements $e$ and $f$ from each other, but we could determine their orders.
\end{example}

In view of Question~\ref{q1} we may ask the following

\begin{question1p}
  For a centerless finite group $G$, by using only the knowledge of the subrack lattice $\mathcal{R}(G)$, can we further refine the process of constructing the cycle forms of the elements of $\Phi(G)$ such that for any element of $\phi_a\in \Phi(G)$ each part of $[\phi_a']$ is a cycle part?
\end{question1p}

For a finite set $S$ we define a partial order relation $\leq$ on the set $\mathcal{P}(S)$ of partitions of $S$ in the following way. Given two elements $[\alpha]$ and $[\beta]$ of $\mathcal{P}(S)$, we say $[\beta] \leq [\alpha]$ if elements of $[\alpha]$ are the unions of some of the elements of $[\beta]$ as sets.

\begin{lem}
  Let $G$ be a finite group whose center is trivial and let $a$ be an element of~$G$. Then the set 
  $$ A_G(a) := \{x\in G\colon [\phi_x']\leq [\phi_a']\} $$
  is an abelian subgroup of $G$.
\end{lem}

\begin{proof}
  Let $x,y\in A_G(a)$. Since the set of fixed points of $\phi_x$ as well as the set of fixed points of $\phi_y$ contains the set of fixed points of $\phi_a$ as a subset, we see that the elements $x$ and $y$ are commuting.

  Next, we shall show that elements of $A_G(a)$ form a subgroup of $G$. Obviously $[\phi_x] = [\phi_{x^{-1}}]$; thus, $A_G(a)$ is closed under taking inverses. To show that $xy\in A_G(a)$ it is enough to prove the following statement. For every $g,h \in G$ the set $B_G(\langle x,y\rangle, g)\setminus B_G(\langle x,y \rangle, h)$ is contained by $B_G(a,g)\setminus B_G(a,h)$.

  Consider the union of $\langle\phi_x\rangle$-orbits of the elements of $B_G(y,g)\setminus B_G(y,h)$. Observe that this set is contained by $B_G(a,g)\setminus B_G(a,h)$ since $[\phi_y']$ and $[\phi_x']$ are less than $[\phi_a']$ by assumption.  Now the statement holds; because, the set $B_G(\langle x,y \rangle, g)\setminus B_G(\langle x,y \rangle, h)$ can be constructed by forming unions of $\langle \phi_x\rangle$-, $\langle \phi_y\rangle$-, or $\langle \phi_{g'}\rangle$-orbits of the base set, where $g'$ is an element of the base set, iteratively.
\end{proof}

We call $A_G(a)$ the \emph{abelian subgroup associated with $a$}. Notice that $A_G(a)$ is a subgroup of the center of the centralizer of $a$ and the order of $a$ (hence, the least common multiple of the cycle lengths of $\phi_a$) divides the order of $A_G(a)$.

It is possible to refine the hypothetical pseudo cycle form $[\phi_a']$ further into subparts. If $a$ and $x$ are two commuting elements of $G$ (if $x$ lies in the centralizer of $a$) and if $F$ is the set of fixed points of $\phi_x$ that are lying in a part $P$ of $[\phi_a']$, then the set $F$ is preserved by $\phi_a$ as well. In consequence, the sets $P\setminus F$ and $F$ should be declared as parts of the cycle form of $\phi_a$. We can justify this assertion in the following way. Let $g \in F$, so that $\phi_x(g) = g$ and $\phi_a(g)\neq g$. If $F$ is not preserved by $\phi_a$, then we may assume $\phi_a(g)\notin F$. Accordingly $\phi_a(g)$ must necessarily be moved by $\phi_x$ and this implies  $\phi_{xa}(g) \neq \phi_a(g)$. On the other hand, $\phi_{ax}(g) = \phi_a(\phi_x(g)) = \phi_a(g)$ which yields a contradiction as $\phi_{xa} = \phi_{ax}$. We shall remark that the cardinality of $F$ must be greater than 1. As otherwise, $\phi_{xa}(g)$ and $\phi_a(g)$ would be two distinct elements.

We define the \emph{hypothetical cycle form} $[\phi_a]$ of $\phi_a$ as the partition of $G$ derived from $[\phi_a']$ by applying the last step explained in the previous paragraph for every $x$ lying in the centralizer of $a$ and for every part $P$ of $[\phi_a']$. 

\begin{lem}\label{lem:order}
  Let $G$ be a finite group  whose center is trivial and let $u$ and $v$ be two elements of $G$. Then $[\phi_v']\leq [\phi_u']$ implies $[\phi_v]\leq [\phi_u]$.
\end{lem}

\begin{proof}
  By assumption an element $w\in G$ lying in the centralizer of $u$, also lies in the centralizer of $v$. Let $P_u$ be a part of $[\phi_u']$ and $F$ be the subset of $P_u$ whose elements are fixed by $w$, so that $P_u\setminus F$ and $F$ are preserved by $\phi_u$. Let $P_v$ be a part of $[\phi_v']$ contained by $P_u$. Obviously, the sets $P_v\setminus F$ and $F\cap P_v$ are preserved by $\phi_v$.
\end{proof}

As an immediate application of Lemma~\ref{lem:order} we may analyze the relations between the elements of $A_G(a)$ via the hypothetical cycle forms of the corresponding inner automorphisms instead of the pseudo cycle forms.

\begin{lem}
  Let $G$ be a finite centerless group and $a$ be an element of $G$. Let $x$ be an element of $A_G(a)$ and $P_1 := [g_1,\dots,g_s]$ be a part of $[\phi_x]$. If $P_1,\dots, P_r$ are the parts of $[\phi_x]$ whose union $P:=\bigcup_{i=1}^rP_i$ is the part of $[\phi_a]$ containing $g_1$, then each of the cycles of $\phi_x$ whose underlying set is a subset of $P$ has the same length $k$, so that $k$ divides $s$.
\end{lem}

\begin{proof}
  Observe that if there are two cycles of $\phi_x$ contained by $P$ whose lengths are different, then for some non-identity power of $\phi_x$ a proper subset of $P$ would be a fixed point set, so that $P$ cannot be a part of $[\phi_a]$ contrary to the assumption. Hence, any two cycles of $\phi_x$ contained by $P$ would have the same length, say $k$. 
\end{proof}

Consider the hypercenter $Z_{\infty} := Z_{\infty}(G)$ of the finite group $G$. By the proof of Theorem~\ref{thm:nil} we know that given the subrack lattice $\mathcal{R}(G)$ of $G$ it is possible to determine an isomorphic copy of $\mathcal{R}(G/Z_{\infty})$. We say $\mathcal{R}(G)$ satisfies the \emph{cycle form condition} if for any pair of elements $a,b$ of $G/Z_{\infty}$ such that $[\phi_a] = [\phi_b]$ there exists a permutation of $G/Z_{\infty}$ interchanging $a$ with $b$, also interchanging $K_{G/Z_{\infty}}(a)$ with $K_{G/Z_{\infty}}(a)$ as sets, and fixing the rest of the elements of $G/Z_{\infty}$ that induces an automorphism of $\mathcal{R}(G/Z_{\infty})$. Now, we can state and prove the main result of this section.

\begin{thm}\label{thm:pnil}
  If $G$ is a $p$-nilpotent group whose subrack lattice satisfies the cycle form condition, and if $H$ is a group with $\mathcal{R}(G)\cong \mathcal{R}(H)$, then $H$ is $p$-nilpotent.
\end{thm}

\begin{proof}
  Let $A$ and $B$ be normal subgroups of a group $K$. Then $K/(A\cap B)$ and $K/A\times K/B$ are isomorphic groups. As a result one can easily observe that the group $K$ has a normal $p$-complement if and only if the quotient $K/Z_{\infty}(K)$ has a normal complement. Let $\widetilde{G} := G/Z_{\infty}(G)$ and $\widetilde{H} := H/Z_{\infty}(H)$. Since the group $G$ is $p$-nilpotent by assumption, the quotient group $\widetilde{G}$ is $p$-nilpotent as well. Also, by the proof of Theorem~\ref{thm:nil} we know that $\mathcal{R}(\widetilde{G})\cong \mathcal{R}(\widetilde{H})$. Therefore, to prove the Theorem it is enough to show that $p$-nilpotence of $\widetilde{G}$ implies the $p$-nilpotence of $\widetilde{H}$. Towards this end we begin by constructing a function $\vartheta$ from $\widetilde{G}\setminus \{1\}$ to the power set of the set of prime numbers such that for a non-identity element $a$ of $\widetilde{G}$ the image is the set of prime divisors of $|\langle a\rangle|$ up to an automorphism of $\mathcal{R}(\widetilde{G})$ induced from a permutation of $\widetilde{G}$. Here we implicitly identify the set of atoms of $\mathcal{R}(\widetilde{G})$ with the elements of $\widetilde{G}$ so that our construction would work.

  As a first step let us compute the hypothetical cycle forms for each element of $\widetilde{G}$ and also the associated abelian groups. Some of those abelian groups are minimal in the sense that all the non-identity elements in a minimal abelian group have the same cycle form. Notice that we know the order of those abelian groups which means we know exactly for how many times the function $\vartheta$ must take a specific value. Since cycle form condition holds by assumption, we can freely ascribe values of $\vartheta$ for the elements of those minimal abelian groups. Observe that we can continue this process step-by-step till we specify the values of $\vartheta$ for all non-identity elements of $\widetilde{G}$ in a complying fashion.

  Next, we want to determine an element of $\mathcal{R}(\widetilde{G})$ corresponding to a Sylow $p$-subgroup of $\widetilde{G}$. Consider the set
  $$ \widetilde{G}_p := \{a\in \widetilde{G}\colon \vartheta(a) = \{p\}\}. $$
  Up to an automorphism of $\mathcal{R}(\widetilde{G})$ the set $\widetilde{G}_p$ is the set of non-identity $p$-power elements of $\widetilde{G}$. Let $p^k$ be the exact power of $p$ dividing $|\widetilde{G}|$. If $|\widetilde{G}_p| + 1$ equals $p^k$, then $\widetilde{G}_p$ is corresponding to the normal Sylow $p$-subgroup and we are done. Otherwise, we may compute $\mathsf{M}(\mathcal{R}(\widetilde{G}))$ and pick an element of this set such that the number of atoms covered by that element is a multiple $p^k$. Recall that the existence of such an element is guaranteed by Lemma~\ref{lem:3.4}. Iterating this process we would arrive at a subrack of $\widetilde{G}$ corresponding to a Sylow $p$-subgroup $P$ of $\widetilde{G}$.

  Now, we know by \cite[Theorem~5.25]{Isa08} the group $\widetilde{G}$ has a normal $p$-complement if and only if any conjugacy class of $P$ is also a conjugacy class in $\widetilde{G}$. Since we can easily determine the conjugacy classes of $\mathcal{R}(P)$ and since $\mathcal{R}(\widetilde{G})$ and $\mathcal{R}(\widetilde{H})$ are isomorphic lattices, we see that $\widetilde{H}$ must be $p$-nilpotent.
  
\end{proof}


\end{document}